\documentclass{amsart}
\usepackage{amsmath}
\usepackage{amsfonts}
\usepackage{amssymb}
\usepackage{graphicx}
\usepackage[mathscr]{eucal}
\newtheorem*{theorem}{Theorem}
\theoremstyle{plain}

\numberwithin{equation}{section}
\begin{document}
\title[Unimodality of $f$-vectors of cyclic polytopes]{Unimodality of $f$-vectors of cyclic polytopes}
\author{L\'{a}szl\'{o} Major}\address[]{L\'{a}szl\'{o} Major \newline\indent Institute of Mathematics \newline\indent Tampere University of Technology \newline\indent PL 553, 33101 Tampere, Finland}\email[]{laszlo.major@tut.fi}
\date{Aug 6, 2011}

\begin{abstract}
Cyclic polytopes are generally known for being involved in the Upper Bound Theorem, but they have another extremal property which is less well known. Namely, the special shape of their $f$-vectors makes them applicable to certain constructions to present non-unimodal convex polytopes. Nevertheless, the $f$-vectors of cyclic polytopes themselves are unimodal.
\end{abstract}
\maketitle

The face vectors of convex polytopes were conjectured to be unimodal, that is, it was conjectured for every $d$-polytope $P$ that there exists some $j$ such that $$f_{-1}\leq f_0\leq \ldots \leq f_j\geq\ldots \geq f_{d-1},$$ where $f_k$ is the number of k-dimensional faces
of the polytope for $-1\leq k\leq d-1$. This conjecture was disproved even for simplicial polytopes (see e.g. Bj\"orner  \cite{bjo}) and even in low dimensions. For a brief historical overview of this topic, see Ziegler \cite{zieg}. However, the conjecture holds for certain polytopes with some restrictions on dimension (see e.g. Werner \cite{wer}). Unimodality may also hold for some families of convex polytopes in any dimension. For example the face vector of the $d$-simplex is clearly unimodal for any $d$.

Ziegler posed the question in \cite{zie} and also in \cite{zieg}, whether the unimodal conjecture holds generally for cyclic polytopes. A partial answer was given by Schmitt \cite{sch} who showed that the $f$-vector of a cyclic $d$-polytope is unimodal if the number of its vertices is sufficiently large compared to $d$. The unimodality is also tested for cyclic polytopes with less than 1000 vertices (\cite{sch} Section 2.3). The remaining range of the number of vertices is covered by the theorem of the present note. 

In fact, we prove a stronger statement, namely, the log-concavity of cyclic polytopes, which implies the unimodality. The sequence $f_{-1},f_0,\ldots,f_d $ is called \emph{log-concave} if $f_{i-1}f_{i+1}\leq f_i^2$ for all $-1<i<d$. 

The Upper Bound Theorem (proved by McMullen \cite{mcm}) provides an upper bound for the $f$-vectors of convex polytopes. This upper bound is attained by the cyclic polytopes. The cyclic polytope $C(v,d)$ is the convex hull of any $v$ points on the moment curve ${(t,t^2,...,t^d):t \in \mathbb{R}}$ in $\mathbb{R}^d$. The combinatorial type of $C(v,d)$ is uniquely determined by $v$ and $d$, therefore its $f$-vector depends only on $v$ and $d$. For the sake of simplicity, let us restrict our attention to even dimensional cyclic polytopes, the odd dimensional case can be dealt with in a similar way, but with the uncomfortable  presence of the floor function. 

The $h$-vector of $C(v,d)$ is defined as follows (see e.g. Gr\"unbaum \cite{grun})   
 \begin{equation}\label{eq18}h_{j}(C(v,d))=\binom{v-d-1+j}{j}, \hspace{2mm}\text{for } \hspace{2mm}0\leq j \leq  \frac d2\end{equation}
 and for $\frac d2 < j \leq d$ we have $h_j(C(v,d))=h_{d-j}(C(v,d))$ from the Dehn-Sommerville equations. The relation between the $h$-vector and $f$-vector is given by the following equations
  \begin{equation}\label{eq17}f_{j}=\sum_{i=0}^d \binom{d-i}{d-j-1}h_i, \hspace{2mm}\text{for } \hspace{2mm}-1\leq j \leq  d-1. \end{equation}
Stanley \cite{sta} formulated the above relation graphically by constructing an integer array, which contains the $f$-vector and $h$-vector of a simplicial polytope. We construct here a similar ''Pascal type triangle'' to illustrate the main idea of the present note. Let $C(v,d)$ be a cyclic polytope with the $h$-vector $h_0,\ldots,h_d$. Let $0\leq k\leq d$. We introduce the following notation
\begin{equation}\label{eqq}\binom{k}{j}_h:=\sum_{i=0}^k \binom{k-i}{k-j-1}h_i(C_d), \hspace{2mm}\text{for } \hspace{2mm}-1\leq j \leq k-1. \end{equation}
In addition, we agree that $\binom{k}{k}_h:=h_{k+1}$ for $0\leq k \leq d-1$ and $\binom{d}{d}_h:=1.$ Furthermore, let $\binom{k}{j}_h:=0$ if $j>k$ or $j<-1.$
Using this notation, we have $f_j=\binom{d}{j}_h$ for the $f$-vector of $C(v,d).$  Using Pascal's rule related to the binomial coefficients we have the recursion \begin{equation}\label{eqq1}\binom{k}{j}_h=\binom{k-1}{j-1}_h+\binom{k-1}{j}_h.\end{equation} 
This recursion allows us to construct a Pascal type triangle such that $\binom{k}{j}_h$ will be the $j$th element of the $k$th row. The $d$th row of this triangle is the $f$-vector of the polytope $C(v,d).$ 

\begin{theorem}The $f$-vectors of cyclic polytopes are log-concave. 
\end{theorem}
 \begin{proof}Using the notation \ref{eqq}, we have to show that the following sequence is log-concave
 $$\binom{d}{-1}_h,\binom{d}{0}_h,\ldots,\binom{d}{d-1}_h,\binom{d}{d}_h.$$
 Let $P(k)$ denote the vector $\Big(\binom{k}{-1}_h,\binom{k}{0}_h,\ldots,\binom{k}{k}_h\Big)$. We shall say that $\binom{k}{j}_h$ is a \emph{dip} of $P(k)$ if $\binom{k}{j}_h^2<\binom{k}{j-1}_h\cdot \binom{k}{j+1}_h$ for some $0\leq j \leq k-1$. It is sufficient to show that $P(d)$ has no dips. If $0<k< \frac d2$, then $P(k)$ is a prefix of a certain row of the Pascal's triangle, therefore $P(k)$ is log-concave (see e.g. Comtet \cite{com}), that is, it has no dips. Now, let us assume that $k\geq \frac d2 $ and $P(k)$ has at least one dip but $P(k-1)$ has no dips. If $\binom{k}{l}_h$ is a dip of $P(k)$, then one can distinguish two different cases: 
 
If $0<l<k-1$, then we  temporarily use  the following notation for the sake of simplicity:
$$a:=\binom{k-1}{l-2}_h, \hspace{2mm}b:=\binom{k-1}{l-1}_h,\hspace{2mm}c:=\binom{k-1}{l}_h,\hspace{2mm}d:=\binom{k-1}{l+1}_h.$$
Since $P(k-1)$ has no dips, $b^2\geq ac$ and $c^2\geq bd.$ On the other hand $\binom{k}{l}_h$ is a dip of $P(k)$, thus we have $(b+c)^2<(a+b)\cdot(c+d)$ by the recursion \ref{eqq1}.
Consequently, $b^2+c^2+bc<ac+ad+bd$, therefore $bc<ad$, which is contradiction, because $b^2\geq ac$ and $c^2\geq bd$ imply that $\frac ab \leq \frac bc$ and $\frac bc \leq \frac ca$, that is, $ad\leq bc$.

In the second case let $l=0$ or $l=k-1$. If $l=k-1$, then we  use  the following notation
$$a:=\binom{k-1}{k-3}_h, \hspace{2mm}b:=\binom{k-1}{k-2}_h,\hspace{2mm}c:=\binom{k-1}{k-1}_h,\hspace{2mm}d:=\binom{k}{k}_h.$$
Since $P(k-1)$ has no dips, $b^2\geq ac$.  On the other hand $\binom{k}{l}_h$ is a dip of $P(k)$, thus we have $(b+c)^2<(a+b)\cdot d$ by the recursion \ref{eqq1}. In addition, we know that $c\geq d$, therefore $b^2+bc+c^2<ac$, which is contradiction, because $b^2\geq ac$. The case $l=0$ similarly leads to a contradiction, therefore the fact that $P(k)$ has a dip implies that $P(k-1)$ has a dip too.

Now, let us assume that the vector $P(d)$ has a dip. Using the above observation recursively, one can show  that the vector $P(\frac d2-1)$ has a dip, which is not possible.
 
The  odd dimensional case can be proved by applying the same method.
 \end{proof}

\end{document}